\documentclass{amsart}
\usepackage{amsmath, amssymb,graphicx,amscd}
\usepackage{graphicx}
\usepackage{xypic}
\usepackage{hyperref, setspace}

\usepackage{colortbl}

\usepackage[all, knot]{xy}
\xyoption{arc}

\newtheorem*{thm-ref}{Theorem}
\newtheorem{thm}{Theorem}[section]

\newtheorem{prop}[thm]{Proposition}
\newtheorem{defn}[thm]{Definition}
\newtheorem{cor}[thm]{Corollary}

\theoremstyle{remark}
\newtheorem{example}{Example}[section]
\newtheorem{remark}{Remark}[thm]

\DeclareMathOperator{\Sign}{sign}
\DeclareMathOperator{\Weight}{weight}

\definecolor{orange}{rgb}{1,0.84,0}
\definecolor{beige}{rgb}{0.93,0.87,.51}
\newcommand{\cy}{\cellcolor{yellow}}

\newcommand{\cb}{\cellcolor{beige}}

\newcolumntype{q}{>{\bfseries}c}

\title[Graphical methods establishing nontriviality of state cycles]{Graphical methods establishing nontriviality of state cycle Khovanov homology classes}
\author{Andrew Elliott}
\thanks{\noindent Partially supported by the National Science 
Foundation  DMS-0706929}
\address{Department of Mathematics\\Rice University\\6100 S. Main Street\\Houston, TX 77005 }
\email{elflord@rice.edu}

\begin{document}

\begin{abstract}
We determine when certain state cycles represent nontrivial Khovanov homology classes by analyzing features of the state graph.  Using this method, we are able to produce hyperbolic knots with arbitrarily many diagonals containing nontrivial state cycle homology classes.  This gives lower bounds on the Khovanov width of knots whose complexity precludes computation of the full homology.
\end{abstract}

\maketitle

\section{Introduction}

\subsection{Khovanov homology and thickness}
Khovanov homology is a bigraded link homology whose graded Euler characteristic is a normalization of the Jones polynomial, introduced in ~\cite{khov}.  As a knot invariant, it is stronger than the Jones polynomial, and it is defined in a combinatorial manner, lending itself to computer calculation.  Using a Krull-Schmidt decomposition, Khovanov showed in ~\cite{khov-patterns} that this homology lies on a certain number of diagonals of slope 2.  The number of diagonals with nontrivial homology is called the (homological) width of the knot, denoted $w_{Kh}(K)$; knots of width 2 are called H-thin, while those of width 3 or more are called H-thick.

There are several motivations for interest in the homological width of a knot.  First of all, for 12 crossing and less knots, Lowrance observed that the reduced Khovanov width matches knot Floer width ~\cite{lowrance-hfk-turaev}.  Secondly, Watson has used width to obstruct when a strongly invertible knot admits a lens space surgery ~\cite{watson-surgery}.  Thirdly, there is a conjectural geometric interpretation of width in terms of the Turaev genus.

Roughly speaking, a Turaev surface for a diagram can be gotten by gluing two extremal states of the diagram along some saddles; for specifics we refer the reader to Section 4 of ~\cite{dasbach-jones_on_graph}.
The genus of this surface is the Turaev genus of the diagram; the minimum such genus over all diagrams is the Turaev genus, denoted by $g_T(K)$.  Manturov ~\cite{manturov-turaev} and Champanerkar-Kofman-Stoltzfus ~\cite{CKS-graphs-and-Kh} have shown that Turaev genus gives an upper bound on (unreduced) width, namely $w_{Kh}(K) \leq g_T(K) + 2$.  Lowrance conjectures that, for prime knots, Turaev genus also gives a lower bound, so that $g_T(K) + 1 \leq w_{Kh}(K) \leq g_T(K) + 2$.  The conjecture is known to hold for prime knots with less than 13 crossings, alternating knots, almost alternating knots, closed 3-braids ~\cite{lowrance-thick-rational}, and adequate knots (unpublished work of Tetsuya Abe). 

We now give a short survey of known results on Khovanov width.  Khovanov ~\cite{khov-patterns} has shown how width is effected by the connect sum operation, namely $w_{Kh}(K_1 \# K_2) = w_{Kh}(K_1) + w_{Kh}(K_2) - 2$.  Asaeda-Przytycki ~\cite{ap-torsion} and Champanerkar-Kofman ~\cite{ck-spanning-tree} have given upper bounds on the width of $k$-almost alternating knots.  Lowrance has found the width of all 3-braids in ~\cite{lowrance-thick-rational}, using rational tangle replacement on width-preserving crossings.  Manolescu and Ozsvath ~\cite{manolescu-2007} have shown that the (unreduced) width of all quasialternating knots is 2.  By fully calculating the homology, Stosic ~\cite{stosic-2006} and Turner ~\cite{turner-2006} have found the width of many torus knots, and Suzuki has found the width of many pretzel knots ~\cite{suzuki-2006}. 

\subsection{The state cycle approach}
State cycles were introduced in ~\cite{elliott-qpmod} as a means of studying homology classes with very simple representatives.  By counting the number of different diagonals with nontrivial state cycle representatives, one can get a lower bound on the width of a knot.  State cycles transform nicely under the operation of quasipositive modification, which let us construct infinite families of H-thick knots which could not be detected by Khovanov's thickness criteria.  However, there were few theorems which could say when a state cycle represented a nontrivial homology class in Khovanov homology, a problem which this paper aims to address.

In particular, we use the structure of subgraphs of the state graph to say when certain state cycles represent nontrivial homology classes.  Describing the subgraphs requires a good deal of extra terminology, so we defer precise statements of the theorems until Sections \ref{sec-all1} and \ref{sec-1even}.  As an application of these theorems, we produce infinite families of hyperbolic knots which have nontrivial state cycles in arbitrarily many different diagonals.  These knots quickly move beyond the current computational limits for directly calculating the Khovanov homology, yet this state cycle information gives a lower bound for their Khovanov width.

\subsection*{Acknowledgements} The author would like to thank Jessica Purcell for helpful conversations.

\section{State cycles and Khovanov homology}

This will be a quick review of material from ~\cite{bn-khov-1} and ~\cite{elliott-qpmod}.  For a more detailed background on this material, we refer the reader to these sources.

\subsection{States and enhanced states}

\begin{figure}[ht!]
\centering
\includegraphics[height=50pt]{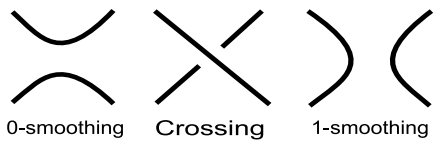}
\caption{Standard smoothing convention for a crossing}
\label{figure-smoothing}
\end{figure}

Given a diagram $ D $ for an oriented link $ L $, one can construct a state for that diagram by replacing every crossing with a choice of smoothing, per the convention in Figure \ref{figure-smoothing}.  The result will be a collection of planar loops, as shown in Figure \ref{figure-smoothing_diagram}.

\begin{figure}[ht!]
\begin{center}
\begin{displaymath}
\xymatrix{
\includegraphics{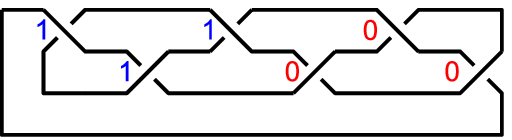} \ar@{~>}[r] & \includegraphics{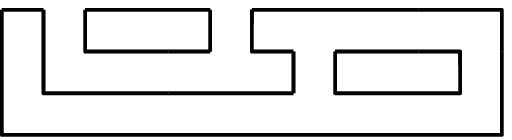} \\
}
\end{displaymath}
\end{center}
\caption{On the left, a choice of smoothings has been assigned to every crossing.  On the right, each crossing has been replaced by its chosen smoothing, giving a state for the diagram.} 
\label{figure-smoothing_diagram}
\end{figure}

Consider such a state where each loop has been marked by one of $v_+, v_-$.  Such a marked state is called an \emph{enhanced state}, and corresponds to a generator of the Khovanov chain complex.  In particular, chains in the chain complex can be viewed as rational linear combinations of these enhanced states.  The precise correspondence is that these $v_+, v_-$ are a basis of a graded rational vector space $V$, and an enhanced states correspond to a tensor product of the vectors marked on its loops, contained in a chain group $C_\sigma$ based on the underlying state $\sigma$.  However, it is very possible to frame everything in terms of these enhanced states, and this is the approach we will take in this paper.

\begin{figure}[ht!]
\begin{center}
\begin{displaymath}
\xymatrix{
\includegraphics{figures/6_3_state-small} \ar@{~>}[r] & \includegraphics{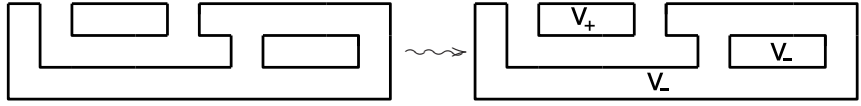} \\
}
\end{displaymath}
\end{center} 
\caption{On the left is a state, on the right is a sample enhanced state based on that state.} 
\label{figure-enhanced_state}
\end{figure}

\subsection{Bigradings of enhanced states}
Khovanov homology is a bigraded homology theory, so each enhanced state has an associated bigrading.  Let $ \alpha_\sigma $ be an enhanced state based on state $ \sigma $.  Let $ n_+ $ and $ n_- $ be the number of positive and negative crossings of $ L $ respectively, following the usual righthanded sign convention.  Let $ v_+ (\alpha) $ and $ v_- ( \alpha ) $ denote the number of loops marked by $ v_+ $ and $ v_- $ in the enhanced state.  Let $h(\sigma)$ denote the number of 1-smoothings in $\sigma$.  Then, the bigrading $ (t, q) $, representing the \emph{homological grading} and \emph{quantum grading} respectively, is given by:

\begin{align}
t(\alpha) &= h(\sigma) - n_- \tag{homological grading}\\
q(\alpha) &= v_+(\alpha) - v_-(\alpha) + h(\sigma) + n_+ - 2 n_-
\tag{quantum grading}
\label{khov-bigrading}
\end{align}

\subsection{Differentials}
Since we will be working with enhanced states, we will describe how the Khovanov differential acts on such an enhanced state.  The Khovanov differential restricts to an action on the state-based chain groups $C_\sigma$, and is broken into little pieces called edge differentials.  Specifically, to each 0-smoothing of a state, there is an associated edge differential going from $C_\sigma$ to $C_{\sigma'}$, where $\sigma'$ is the state where this 0-smoothing is changed into a 1-smoothing.

This change from a 0-smoothing to a 1-smoothing can be viewed as a cobordism, a pair of pants which either merges two loops from $\sigma$ into a single loop in $\sigma'$, or splits a single loop from $\sigma$ into two loops in $\sigma'$.  The associated edge differential acts as the identity on loops not changed by this cobordism, and acts as in Table \ref{edge-differential} on the loops of the cobordism:

\begin{table}[ht!]
\begin{displaymath}
\xymatrix@R=2pt{ \bigcirc \phantom{\otimes} \bigcirc \ar[r]^{\quad \mu} & \bigcirc & \qquad & \bigcirc \ar[r]^{\Delta \qquad \quad} & \bigcirc \phantom{\otimes} \bigcirc \phantom{+ \bigcirc \otimes \bigcirc}\\
v_+ \otimes v_+ \ar@{|->}[r] & v_+ & \qquad & v_+ \ar@{|->}[r] & v_+ \otimes v_- + v_- \otimes v_+ \\
v_+ \otimes v_- \ar@{|->}[r] & v_- & \qquad & v_- \ar@{|->}[r] & v_- \otimes v_- \phantom{+ v_- \otimes v_-}\\
v_- \otimes v_+ \ar@{|->}[r] & v_- \\
v_- \otimes v_- \ar@{|->}[r] & 0}
\end{displaymath}
\caption{The edge differential, $d_e$ takes one of the above forms depending on whether the associated cobordism takes two loops to one, or vice versa.  Outside of the changed part of the state, the edge differential acts as the identity.\label{edge-differential}}
\end{table}

These edge differentials are then bundled into a differential acting on $C_\sigma$, with a sprinkling of negative signs so that ``squares'' of edge differentials anticommute.  In this paper, we will typically be able to ignore these signs; full details of how these signs are chosen can be found in ~\cite{bn-khov-1}.

\subsection{State cycle terminology}

Enhanced states comprise the generators of the Khovanov chain complex.  But sometimes, a single such enhanced state turns out to be a cycle representative.  We will be especially interested in the situation that such an enhanced state represents a nontrivial homology class.

\begin{defn}
We say that an enhanced state $\alpha$ is a \emph{state cycle} if the associated element of the chain complex is a cycle; namely, $d(\alpha) = 0$.  We say that a state cycle is \emph{nontrivial} if it represents a nontrivial homology class.
\label{def-state_cycle}
\end{defn}

\begin{figure}[ht!]
\centering
\includegraphics[height=50pt]{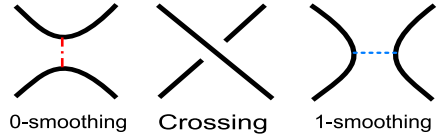}
\caption{Smoothings marked with traces of the crossings. 0-smoothings will be red, dot-and-dashed lines; 1-smoothings will be blue, dotted lines.}
\label{figure-trace}
\end{figure}

To talk about whether such an enhanced state is a homology cycle, or represents a nontrivial homology class, we have to see how the edge differentials act.  In order to analyze the action of these edge differentials, it is helpful to record not just the smoothings of $\sigma$, but also the traces of the crossings, as shown in Figure \ref{figure-trace}. The \emph{trace} of a crossing is a shadow to show where a crossing has been smoothed to get the state $\sigma$, and represents where an edge differential either enters or exits a chain generator based on $\sigma$.  See Figure \ref{figure-trace-diff} for a schematic of this relation.  A state with all its tracings marked is said to be a \emph{traced state}; an enhanced state with all its tracing marks is said to be an \emph{enhanced trace state}, or \emph{ET state}.  See Figure \ref{figure-tracing_diagram} for an example of an ET state.

\begin{figure}[ht!]
\centering
\includegraphics{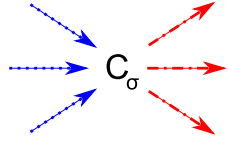}
\caption{Schematic for how traces correspond to edge differentials.  Red dot-and-dash edge differentials exit the chain group, and come from the red dot-and-dash 0-traces; blue dotted edge differentials enter the chain group, and come from the blue dotted 1-traces.}
\label{figure-trace-diff}
\end{figure}

\begin{figure}[ht!]
\begin{center}
\begin{displaymath}
\xymatrix{
\includegraphics{figures/6_3_diagram_numbered-small} \ar@{~>}[r] & \includegraphics{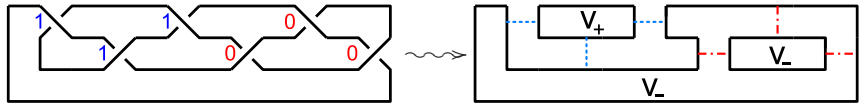} \\
}
\end{displaymath}
\end{center}
\caption{On the left, a choice of smoothings has been assigned to every crossing.  On the right, an ET state corresponding to $v_+ \otimes v_- \otimes v_-$ is shown.} 
\label{figure-tracing_diagram}
\end{figure}

Let us introduce some further terminology to discuss these traces of crossings.  We say that a trace is a \emph{mergetrace} or \emph{pinchtrace} if the trace connects two or one loop in $\sigma$, respectively.  See Figure \ref{figure-tracetypes} for some examples.  The terminology is meant to suggest that when the crossing associated to a mergetrace is changed to the opposite smoothing, the two loops joined by the mergetrace are merged together; similarly, when the crossing associated to a pinchtrace is changed to the opposite smoothing, the original loop is pinched into a pair of loops.

\begin{figure}[ht!]
\centering
\includegraphics{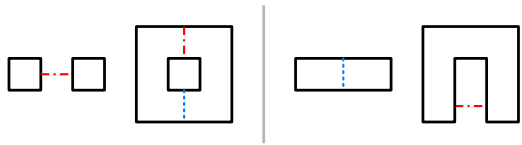}
\caption{The left half shows examples of mergetraces; the right half shows examples of pinchtraces.}
\label{figure-tracetypes}
\end{figure}

We can further differentiate traces by keeping track of which kind of smoothing they are associated to in $\sigma$.  If a mergetrace is associated to a crossing that has been 0-smoothed in $\sigma$, we say it is a \emph{0-mergetrace}, and so on.  A state is said to be \emph{0-merging} if every 0-trace in the state is a mergetrace.  Similarly, a state is said to be \emph{1-merging} if every 1-trace is a mergetrace, and so on for \emph{0-pinching} and \emph{1-pinching}.  This leads to the definition of an \emph{adequate state} as a state that is both 0-merging and 1-merging.  

Building off the above definition, a diagram is said to be \emph{adequate} if its all-0 state and its all-1 state are both adequate as states, \emph{+ adequate} if its all-0 state is adequate as a state, and \emph{- adequate} if its all-1 state is adequate as a state.  Similarly, a link is said to be any of the above if it admits a diagram with that property.

As a final bit of terminology useful in discussing ET states, we say that a loop in $\sigma$ is \emph{0-tracing} if it is touched by a 0-trace in the traced state $\sigma$.  The \emph{1-block} of $\sigma$ is the set of loops which are 1-tracing, but not 0-tracing.  We are now ready to describe the basic condition for an enhanced state to be a state cycle:

\begin{prop}[~\cite{elliott-qpmod}]
\label{zeromergingcycle}
Let $\alpha_\sigma$ be an ET state.  Then $\alpha_\sigma$ is a state cycle if and only if $\sigma$ is 0-merging and every 0-tracing loop of $\alpha_\sigma$ is marked by $v_-$.
\end{prop}

\begin{example}
\label{sigma0}
Suppose $D$ is a + adequate diagram.  Let $\sigma_0$ denote the all-0 state of $D$.  By definition, $\sigma_0$ is an adequate state, and in particular 0-merging.  Let $\alpha_0$ be the ET state of $\sigma_0$ where every loop is marked by $v_-$, as shown in Figure \ref{figure-example_sigma0}.  Proposition \ref{zeromergingcycle} tells us that $\alpha_0$ is a state cycle, but because $\sigma_0$ has no 1-traces, there are no differentials entering $\sigma_0$.  Therefore $\alpha_0$ actually represents a nontrivial homology class, one of minimal homological and quantum grading.
\end{example}

\begin{figure}[ht!]
\begin{center}
\begin{displaymath}
\xymatrix{
\includegraphics{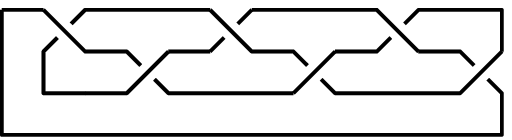} \ar@{~>}[r] & \includegraphics{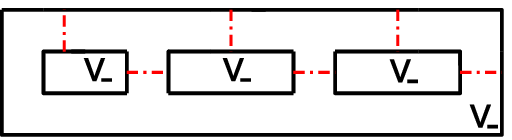} \\
}
\end{displaymath}
\end{center}
\caption{On the left is a + adequate diagram of $6_3$.  On the right, the all-0 ET state $\alpha_0$ is shown.} 
\label{figure-example_sigma0}
\end{figure}

\begin{example}
\label{sigma1}
Suppose $D$ is a - adequate diagram.  Let $\sigma_1$ denote the all-1 state of $D$.  Again, $\sigma_1$ is an adequate state, and so is both 0- and 1-merging.  Let $\alpha_1$ be the ET state of $\sigma_1$ where every loop is marked by $v_+$, as shown in Figure \ref{figure-example_sigma1}.  The set of 0-tracing loops is empty because this is the all-1 state, so Proposition \ref{zeromergingcycle} trivially holds and $\alpha_1$ is a state cycle.

However, since every other state has fewer 1-traces than $\sigma_1$, $C_{\sigma_1}$ is the only chain group of its height $h$ in the Khovanov chain complex for $D$.  This means that every edge differential of height $h-1$ targets only $C_{\sigma_1}$, allowing us to restrict to a single edge differential for each chain group of height $h-1$.  Because every 1-trace of $\sigma_1$ is a mergetrace, every incoming edge differential must lie in the $\Delta$ half of Table \ref{edge-differential}.  By inspection, no edge differential contains a term marking $ v_+ $ on every loop of an ET state for $ \sigma_1 $.  So, no linear combination of terms can have boundary equal to $\alpha_1$.  Therefore, $\alpha_1$ does not lie in $im(d)$ for the Khovanov differential, and hence must represent a nontrivial homology class of $Kh(L)$.
\end{example}

\begin{figure}[ht!]
\begin{center}
\begin{displaymath}
\xymatrix{
\includegraphics{figures/6_3_diagram} \ar@{~>}[r] & \includegraphics{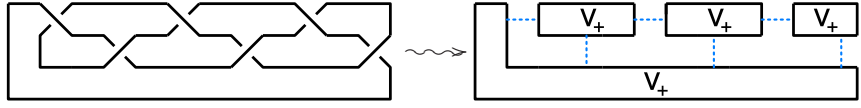} \\
}
\end{displaymath}
\end{center}
\caption{On the left is a - adequate diagram of $6_3$.  On the right, the all-1 ET state $\alpha_1$ is shown.} 
\label{figure-example_sigma1}
\end{figure}

\begin{example}
\label{olga-class}
Let $D$ be a braid diagram, and $\sigma$ be the oriented resolution, the state where every positive crossing has been 0-smoothed, and every negative crossing has been 1-smoothed.  This smoothing choice is such that the loops of the resulting state are just the strands of the braid, a concentric set of circles, with each trace going between two strands.  So, each trace is a mergetrace, and the state is adequate.  Let $ \psi $ be the ET state where every loop of $\sigma$ has been marked by $v_-$, as shown in Figure \ref{figure-example_olga}.  Proposition \ref{zeromergingcycle} tells us that $\psi $ is a cycle in Khovanov homology, but more can be said:  Plamenevskaya has shown ~\cite{olga} that $[ \psi ]$ is a transverse knot invariant, the Plamenevskaya class for $Kh(L)$.
\end{example}

\begin{figure}[ht!]
\begin{center}
\begin{displaymath}
\xymatrix{
\includegraphics{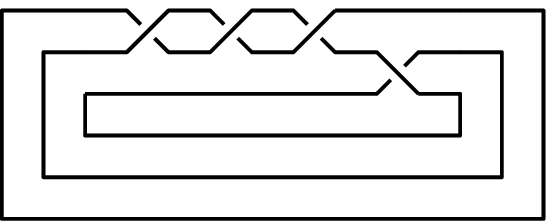} \ar@{~>}[r] & \includegraphics{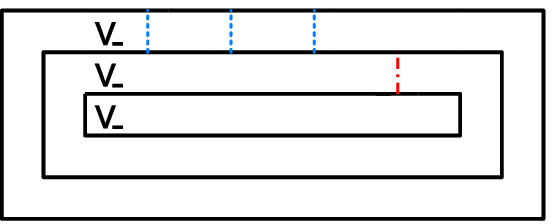} \\
}
\end{displaymath}
\end{center}
\caption{On the left is a braid diagram of the negative trefoil.  On the right, the ET state for the oriented resolution is shown, a representative for the Plamenevskaya class.} 
\label{figure-example_olga}
\end{figure}

\subsection{The state graph and classification of nontrivial state cycles}

State cycles representing a nontrivial homology class have a fairly restricted form.  To describe this form precisely, we will need to introduce the state graph and some related graph theory terms.

\begin{defn}
Given a state $\sigma$, the associated \emph{state graph} $\Gamma_\sigma$ is constructed by taking the loops of $\sigma$ to be vertices, and the traces to be edges.
\end{defn}

\begin{remark}
In general this is really a pseudograph, because an edge may start and end at the same vertex if the state is not adequate, and two vertices may be joined by multiple edges.  When the state is adequate, the state graph is an honest multigraph.
\end{remark}

The state graph turns out to be a useful tool for describing conditions for triviality and nontriviality of a state cycle.  For example, consider the following definition:

\begin{defn}
We say that a state is \emph{even} if every circuit in its state graph $\Gamma_\sigma$ has even length.  Otherwise, we say that the state is \emph{odd}.
\label{def-evenstate}
\end{defn}

\begin{remark}
A pinchtrace gives a closed path of length 1 from the loop it joins to itself.  So, an even state is also an adequate state.
\end{remark}

\begin{remark}
A graph theory consequence is that if $\sigma$ is even, $\Gamma_\sigma$ is 2-colorable.  Namely, there is an assignment of a color to every vertex, using only 2 colors, so that two vertices joined by an edge have distinct colors.  One can view a 2-coloring as a sign choice for something associated to each vertex, so that one color represents ``+'' and the other represents ``-''.
\end{remark}

\begin{example}
The primary example of an even state is the Seifert state $\sigma_s$, the state gotten by choosing the smoothing for each crossing corresponding to that used in Seifert's algorithm.  $\sigma_s$ is also called the oriented resolution, because the smoothing choice is the one consistent with the orientation of the link:  positive crossings are 0-smoothed, and negative crossings are 1-smoothed, as shown in Figure \ref{figure-oriented-smoothing}.  For an example of a Seifert state and graph, see Figure \ref{figure-seifert-stategraph}.

\begin{figure}[ht!]
\centering
\includegraphics{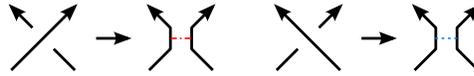}
\caption{The smoothing choice for Seifert's algorithm is the smoothing consistent with orientation.}
\label{figure-oriented-smoothing}
\end{figure}

The reason the Seifert state is even has to do with the surface $\Sigma$ associated to this state by Seifert's algorithm, formed by gluing disks to each loop of the state, and twisted bands to every trace which match the original crossing.  An odd path of traces means there is an annulus with an odd number of twists contained in the surface (gotten by following the twisted bands associated to the traces of the path), contradicting orientability of the Seifert's algorithm surface.
\label{ex-seifert-state}
\end{example}

\begin{figure}[ht!]
\centering
\includegraphics{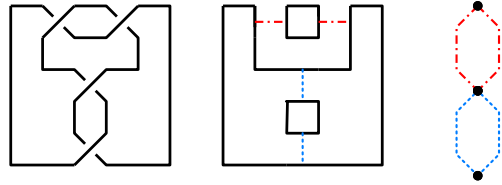}
\caption{On the left is a diagram of the figure 8.  In the middle is its Seifert state.  On the right is the state graph for this state.  By inspection this is an even state.}
\label{figure-seifert-stategraph}
\end{figure}

The fact that even states can be 2-colored turns out to be quite useful:  using this 2-coloring turns out to be the key step in showing there are two nontrivial state cycles based on the all-1 state when it is even (Theorem \ref{thm-classification-exception}).  But, for general state cycles, we will want to consider similar notions of even and oddness, restricted to the 1-block.  To describe these notions, we will examine a particular subgraph of the state graph:

\begin{defn}
Given a state $\sigma$, the associated \emph{1-block graph} $\Gamma_1$ is constructed by taking the loops of the 1-block of $\sigma$ to be vertices, and the 1-traces to be edges.
\end{defn}

In general, $\Gamma_1$ will not be connected, and its connected components turn out to be the natural setting for describing restrictions on the 1-block.

\begin{defn}
A \emph{connected component of the 1-block} refers to the set of loops from a connected component of $\Gamma_1$.  We say that a connected component of the 1-block is \emph{even} if every circuit in that component of $\Gamma_1$ has even length, and \emph{odd} otherwise.
\end{defn}

Now we are ready to state the classification theorem precisely.  Restrictions (S1) and (S2) deal with the underlying state, and restrictions (L2)-(L4) spell out the full obstructions to having loops in the 1-block marked by $v_-$:

\begin{thm}[State Cycle Classification, ~\cite{elliott-qpmod}]
For a state cycle $\alpha$ based on state $\sigma$ to represent a nontrivial homology class in Khovanov homology, it must satisfy the following restrictions:
\begin{itemize}
\item[(S1)] $\sigma$ must be 0-merging.
\item[(S2)] $\sigma$ can have no 1-pinchtraces touching loops in its 1-block.
\item[(L1)] 0-tracing loops of $\alpha$ must be marked by $v_-$.
\item[(L2)] No pair of loops both marked by $v_-$ can be joined by only 1-traces.
\item[(L3)] Every loop in an odd connected component of the 1-block must be marked by $v_+$.
\item[(L4)] At most one loop in an even connected component of the 1-block may be marked by $v_-$; all other loops in that component must be marked by $v_+$.
\end{itemize}
Furthermore, other than condition (L2), changing which loop is marked by $v_-$ for a fixed even component from (L4) only changes the sign of the resulting homology class.
\label{thm-statecycle-classification}
\end{thm}

This theorem gives a good picture of what a nontrivial state cycle should look like, but it does not guarantee that such a candidate actually represents a nontrivial homology class.  We address this issue in the next two sections, providing some graph-based conditions which guarantee nontriviality of a state cycle.

\section{A graph obstruction for even all-1 states}
\label{sec-all1}

In this section, we examine the nontriviality of a very specific kind of state cycle, an even all-1 state cycle with a single loop marked by $v_-$.  While this initial case is restrictive, it serves as a good model for the more general theorem of Section \ref{sec-1even}.

\begin{thm}
Suppose that the all-1 state $\sigma$ for a nonsplit link diagram is even.  Then there are always two nontrivial homology classes represented by state cycles based on $\sigma$:  one is uniquely represented by marking a $v_+$ on every loop, and the other is represented up to sign by marking a single loop of choice by $v_-$ and all other loops by $v_+$. 
\label{thm-classification-exception}
\end{thm}

Before going through the full proof, let's go through a small example.

\begin{figure}[ht!]
\centering
\includegraphics{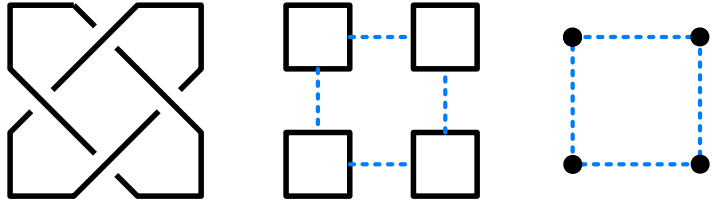}
\caption{On the left side is a diagram for the mirror of Solomon's knot.  In the middle is the all-1 state, and on the right is the all-1 state's state graph.}
\label{figure-solomon-all1}
\end{figure}

\begin{example}
\label{example-solomon}
Consider the mirror of Solomon's knot (which is really a link).  As shown in Figure \ref{figure-solomon-all1}, its all-1 state $\sigma$ is even, with a square state graph.  We already know by Example \ref{sigma1} that marking every loop of this state by $v_+$ gives a state cycle representing a nontrivial homology class.  So, let's number the four loops of the state and let $\alpha_i$ be the enhanced state where loop $L_i$ is marked by $v_-$ and all other loops are marked by $v_+$.  Similarly, number the traces so that trace $T_i$ goes from $L_i$ to $L_{i+1}$, with $T_4$ going between $L_4$ and $L_1$.  Finally, let $\beta_i$ be the enhanced state where trace $T_i$ has been changed from 1 to 0, and all loops are marked by $v_+$.  See Figure \ref{figure-solomon-beta}.

\begin{figure}[ht!]
\centering
\includegraphics{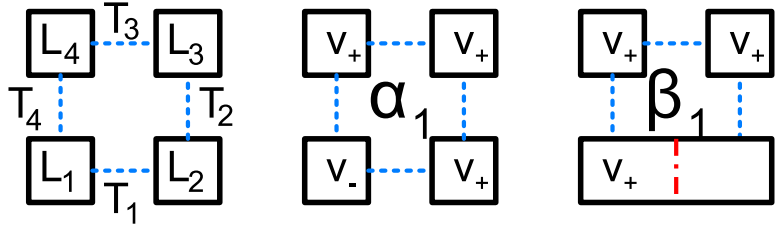}
\caption{On the left side is the numbering convention for loops and traces of the all-1 state.  In the middle is an enhanced state for $\alpha_1$, and on the right is an enhanced state for $\beta_1$.}
\label{figure-solomon-beta}
\end{figure}

This setup was used many times in Section 4 of ~\cite{elliott-qpmod}, and it is easy to verify that $d(\beta_i) = \alpha_i + \alpha_{i+1}$, up to sign.  An alternating sum of such terms creates a telescoping series, so that $\alpha_i \pm \alpha_j = \sum (-1)^{k+1} d(\beta_k)$.  This means that, up to sign, the $\alpha_i$ represent the same homology class.

Now, let's try to show that $\alpha_1$ does not lie in the image of $d$.  By assumption, $\sigma$ is an even state, and hence adequate, so that every trace is a mergetrace.  So, every edge differential entering $C_\sigma$ takes the $\Delta$ form of Table \ref{edge-differential}.  Since there is only a single $v_-$ in $\alpha_1$, there is no chance for a $v_- \otimes v_-$ term to be nontrivially involved in any linear combination resulting in $\alpha_1$.  Thus, we may restrict our attention to the other kind of $\Delta$ output, that of the form $v_+ \otimes v_- + v_- \otimes v_+$, which comes from the image of the $\beta_i$ under the differential.  In particular, $[ \alpha_1 ]$ is trivial if and only if $ \alpha_1$ is some nontrivial linear combination of the $d(\beta_i)$.

Suppose $\alpha_1 = \sum a_i d(\beta_i)$.  Approaching this via linear algebra, we get a simple system of equations, knowing that the coefficient of all $\alpha_i$ other than $i=1$ must be zero:

\begin{align*}
a_1 + a_2 &= 0 \\
a_2 + a_3 &= 0 \\
a_3 + a_4 &= 0 \\
\end{align*}

Consequently, $a_1 + a_4 = 0$.  But, that sum was the coefficient of $\alpha_1$ from the sum, so $\alpha_1$ cannot be realized as such a linear combination. 
 
This works for small cases, but seems to hide the combinatorics of the state:  ideally a proof of nontriviality would exploit this added structure to take advantage of the hypothesis.  One way to take this into account is to associate such a linear combination to a weighted state graph, where the edges of the state graph have weight $a_i$.

Note that $d(\beta_i)$ contributes $a_i$ copies of $\alpha_i$ and $\alpha_{i+1}$.  By interpreting the edge $T_i$ as the $d(\beta_i)$, and the vertices bounding that edge as the associated state cycles $\alpha_i$ and $\alpha_{i+1}$, we can interpret the linear algebra constraint on the coefficients as the condition that the sum of the weights around every vertex other than $L_1$ must be 0.

In our case, every vertex is adjacent to exactly two edges.  So, the weight condition means that the two weights must have the same magnitude and opposite sign.  As we travel around the square, the weights follow this pattern:  alternating in sign, same magnitude.  So, let's pick an edge adjacent to $L_1$ to start a traversal, and propagate this condition on the edge weights until we return to $L_1$.  Starting with $a_1$, the other weights must then be $-a_1, a_1, -a_1$, and the sum of the two weights of the edges around $L_1$ are then $a_1 - a_1=0$.  This same pattern holds any time we have such an even closed paths based at $L_1$.  See Figure \ref{figure-solomon-propagation}.

\begin{figure}[ht!]
\centering
\includegraphics{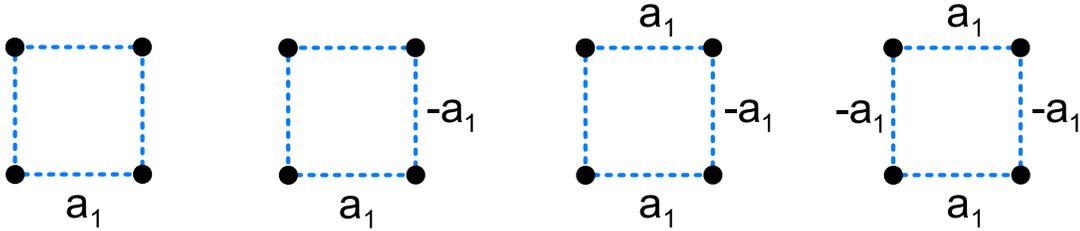}
\caption{From left to right, a propagation of the weight condition over a traversal of a closed path based at $L_1$.}
\label{figure-solomon-propagation}
\end{figure}

The weighted graph interpretation may be overkill for an easy case like this, but the idea is to use this interpretation to take advantage of geometric information in solving the associated linear algebra problem.  The general case involves a more complicated graph, but it turns out the fact that the graph is 2-colorable is enough to force the linear combination to be trivial, with this setup.
\end{example}

So, let's generalize the weighted graph approach of Example \ref{example-solomon} to the general case of Theorem \ref{thm-classification-exception}.

\begin{proof}[Proof of Theorem \ref{thm-classification-exception}]
The first nontrivial homology class associated to this even all-1 state comes from marking every loop by $v_+$.  By the discussion of Example \ref{sigma1}, we know this enhanced state represents a nontrivial homology class.

The new content of this theorem comes in dealing with the enhanced states which mark a single loop by $v_-$ and all other loops by $v_+$.  The first claim is that our choice of loop to mark only changes the sign of the associated homology class.

Label the loops of the traced all-1 state $\sigma$ by $L_1$ to $L_n$, and the traces by $T_1$ to $T_m$.  Let $\alpha_i$ be the enhanced state where loop $L_i$ is marked by $v_-$ and all other loops are marked by $v_+$.  Let $\sigma_j$ be the state where $\sigma$ is modified by changing trace $T_j$ from 1 to 0.  Since $\sigma$ is adequate, $\sigma_j$ will have one loop $M_j$ which was merged by the change of smoothings, and otherwise will match the loop structure of $\sigma$.  Let $\beta_j$ be the enhanced state where $\sigma_j$ is the underlying state and every loop is marked by $v_+$.

Given $i$ and $j$, there is some path of traces from $L_i$ to $L_j$ in $\sigma$ because $\sigma$ is a state of a nonsplit link.  Let this path of traces be $T'_1, \ldots, T'_p$, and trace $T'_h$ join loop $L'_h$ to loop $L'_{h+1}$.  Similarly, let $\alpha'_h$ and $\beta'_h$ be the obvious analogues in this renumbering of loops and traces.  By this setup, $d(\beta'_h) = \alpha'_h + \alpha'_{h+1}$, and we have the telescoping series $\sum (-1)^{h+1} d(\beta'_h) = \alpha'_1 \pm \alpha'_p = \alpha_i \pm \alpha_j$.  So, up to sign any choice of $\alpha_i$ represents the same homology class. 

The meat of the proof is showing that this homology class represented by $\alpha_1$ is nontrivial.  As was the case in Example \ref{example-solomon}, this boils down to showing that no nontrivial linear combination of the $d(\beta_i)$ can equal $\alpha_1$.  To show this, we will reinterpret such a linear combination as a weighted state graph $\Gamma_\sigma$.  If $\alpha_1 = \sum a_i d(\beta_i)$, we will place weight $a_i$ on the edge associated to trace $T_i$.  The sum of weights about the vertex for $L_i$ corresponds to the coefficient of $\alpha_i$ in this linear combination, so the sum of weights about every vertex except $L_1$ must be zero.

Because $\sigma$ is even, the state graph is 2-colorable.  Since the state graph is path-connected, a choice of color on one vertex determines the 2-coloring, so let's use the colors ``+'' and ``-'', and assume that vertex $L_1$ is colored ``-''.  Let $\Sign(L_i)$ denote this color, overloaded to be an actual sign choice.

Now, let's use this sign choice to take an alternating sum of the weights around every vertex.  Let $\Weight(L_i)$ denote the sum of the weights around vertex $L_i$.  We claim that $\sum \limits_{i=1}^{n} \Sign(L_i) \Weight(L_i) = 0$.  To see this, note that if we ignore the sign choice, every edge's weight is counted exactly twice, once for each vertex at the endpoints of the edge.  But, the 2-coloring guarantees that in this weighted sum, the two occurences of weight show up with opposite sign, since if two vertices are joined by an edge in the 2-coloring, they have opposite colors, or opposite signs in this color scheme.

We chose the sign of the $L_i$ vertex to be negative, which means that we can bring it to the other side of this equality to conclude that $\Weight(L_1) = \sum \limits_{i=2}^{n} \Sign(L_i) \Weight(L_i)$.  But, $\Weight(L_i) =0$ for every $i$ other than 1 by construction, so $\Weight(L_1) =0$.  This means that the coefficient of $\alpha_1$ in the original linear combination $\sum a_i d(\beta_i)$ must be zero.  In other words, $\alpha_1$ does not lie in the image of the differential, and represents a nontrivial homology class.
\end{proof}

\section{1-even, 1-isolated state cycles}
\label{sec-1even}

In fact, the proof of Theorem \ref{thm-classification-exception} can be generalized to a larger class of state cycles.  Towards this end, we need to consider yet another subgraph of the state graph.

\begin{defn}
Let $\sigma$ be a state in a diagram.  Let $\Lambda_1$ denote the graph whose vertices correspond to the 1-tracing loops of $\sigma$ (i.e. loops touched by at least one 1-trace).  We say that $\sigma$ is \emph{1-even} if $\Lambda_1$ has only even circuits.  If $\alpha$ is any state cycle based on $\sigma$, we say that $\alpha$ is \emph{1-even} if $\sigma$ is 1-even.
\end{defn}

\begin{remark}
This definition ensures that a 1-even state cycle is based on an adequate state.  The state of any state cycle is 0-merging, per Proposition \ref{zeromergingcycle}, and the 1-even condition tells us that every 1-trace must also be a mergetrace.  So, every trace is a mergetrace, and the state is adequate. 
\end{remark}

\begin{defn}
Let $\alpha$ be a state cycle based on state $\sigma$.  We say that $\alpha$ is \emph{1-isolated} if every connected component of $\Lambda_1$ has at most one loop marked by $v_-$.
\end{defn}

\begin{thm}
\label{thm-1even_1isolated}
Let $\alpha$ be a state cycle which is 1-even and 1-isolated.  Then $\alpha$ represents a nontrivial homology class.
\end{thm}

\begin{proof}
The idea here is that with this setup, each connected component of $\Lambda_1$ looks exactly like the state graph we considered in Theorem \ref{thm-classification-exception}.  Let the underlying state of $\alpha$ be $\sigma$.  The 1-traces of $\sigma$ are the source of edge differentials that target $C_\sigma$, the chain group in the Khovanov chain complex in which $\alpha$ lies.  The difference here is that since $\sigma$ is no longer the all-1 state, there will be other chain groups of the same height as $C_\sigma$.  This means that, for some 1-trace $T$ of $\sigma$, if we have $d_T(\gamma) = \alpha$, $d(\gamma)$ may also have nonzero image in other chain groups besides $C_\sigma$.

However, our strategy here is to simply restrict our attention to the component of the differential which targets $C_\sigma$, which we will denote $d_\sigma$.  What we will show is that there is no $\gamma$ so that $d_\sigma(\gamma) = \alpha$.  Consequently, there can be no $\gamma$ so that $d(\gamma) = \alpha$, and $[ \alpha ] \neq 0$. 

Given a 1-trace $T$ of $\sigma$, let $\sigma_T$ be the state where the $T$ is changed from a 1-trace to a 0-trace, and all other smoothings remain the same, and let $d_T$ denote the edge differential $d_T: C_{\sigma_T} \longrightarrow C_\sigma$.  Thinking through the definitions, it is easy to see that $d_\sigma = \sum \limits_{\text{$T$ 1-trace}} (-1)^{|T|} d_T$.  Since we are only considering this restricted differential, we can safely ignore the signs in front of these edge differentials:  we can pick a negative basis for $C_{\sigma_T}$ any time the sign for $d_T$ should be negative, and end up with the same result as if we had ignored the sign and chosen the positive basis.  So, we will view $d_\sigma$ as simply $\sum \limits_{\text{$T$ 1-trace}} d_T$.

Because $\alpha$ is 1-even, we know that every 1-trace is a mergetrace (a 1-pinchtrace would give an odd closed path in $\Lambda_1$).  So, we know that every edge differential $d_T$ will take the $\Delta$ form.  In general, $\Delta$ can locally return either $v_- \otimes v_-$ or $v_+ \otimes v_- + v_- \otimes v_+$, where the tensor here corresponds to a 1-trace joining the two loops in $\sigma$.  However, because $\alpha$ is 1-isolated, each connected component of $\Lambda_1$ has at most one loop marked by $v_-$, so that the $v_- \otimes v_-$ form of $\Delta$ cannot be involved in any potential boundary for $\alpha$.  So, the only terms of interest will have the local form $v_+ \otimes v_- + v_- \otimes v_+$.

Label the 1-traces of $\sigma$ by $T_1$ through $T_n$, and let $\beta_i$ denote the enhanced state based on $\sigma_{T_i}$ where the merged loop, and every loop in the connected component of $\Lambda_1$ in which this loop lies, are marked by $v_+$, and every other loop is marked as in $\alpha$.  We have argued that the only possible linear combination whose $d_\sigma$ image can be $\alpha$ must take the form $\sum \limits_{i=1}^n a_i \beta_i$.

Now, we may proceed exactly as in Theorem \ref{thm-classification-exception}, building a weighted graph version of $\Lambda_1$, where the weights of the edges are the $a_i$, and the sum of weights around all but one vertex of every connected component of $\Lambda_1$ must be 0.  Since $\Lambda_1$ is even, it is 2-colorable, and the alternating sum of weights arguments lets us conclude again that the sum of weights around \emph{every} vertex is 0.  Consequently, $\sum \limits_{i=1}^n d_\sigma (a_i \beta_i) = 0$, so that $\alpha$ cannot lie in the image of $d_\sigma$.  It then follows that $[ \alpha ] \neq 0$, as desired.
\end{proof}

\begin{figure}[ht!]
\centering
\includegraphics{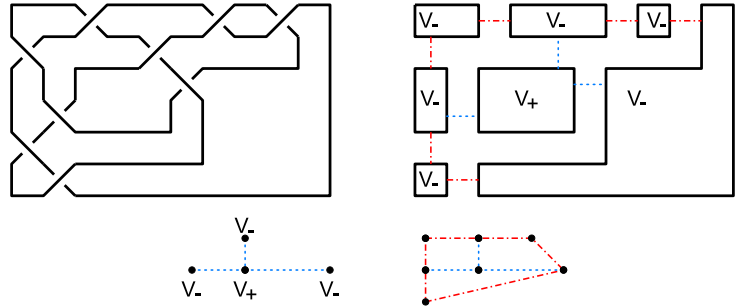}
\caption{Example of a 1-even state cycle which is not 1-isolated, and represents a trivial homology class.  On the top left is the diagram of the knot; on the top right is the state cycle in question.  The bottom right shows the state graph for the state cycle, and the bottom left shows $\Lambda_1$, with the vertices marked as in the state cycle.  Clearly $\Lambda_1$ is 1-even but not 1-isolated.}
\label{figure-example_1even_trivial}
\end{figure}

\begin{example}
\label{example-1even_trivial}
As an example of why the 1-isolated condition is needed, Figure \ref{figure-example_1even_trivial} shows a state cycle which is 1-even, but not 1-isolated.  Calculation shows that the associated homology class is trivial.
\end{example}

This theorem also has some overlap with Theorem 5.6 of ~\cite{elliott-qpmod}, since the Seifert state is even, and hence 1-even.  However, the state cycles of this theorem are not required to be 1-isolated.  We recall the theorem below, and then give an example of a Seifert state satisfying it which is not 1-isolated.

\begin{thm-ref}[~\cite{elliott-qpmod}]
Suppose $\alpha$ is a state cycle associated to the Seifert state, and its quantum grading is $s-1$.  If $\alpha$ has a single loop in its 1-block, and that loop is marked by $v_+$, then $\alpha$ represents a nontrivial homology class.
\end{thm-ref}

\begin{figure}[ht!]
\centering
\includegraphics{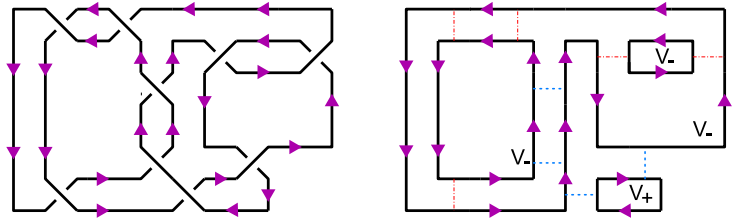}
\caption{An even state cycle satisfying Theorem 5.6 of ~\cite{elliott-qpmod} which is not 1-isolated}
\label{figure-9_42-seifert-oet}
\end{figure}

\begin{example}
\label{example-9_42-seifert}
The $s$ invariant for $9_{42}$ is 0.  For the usual Rolfsen diagram of $9_{42}$, the state cycle associated to the Seifert state has quantum grading $-1 = s-1$ and a single loop in its 1-block, as seen in Figure \ref{figure-9_42-seifert-oet}.  So, by the theorem, this state cycle is nontrivial.  But, while this state is 1-even, it is not 1-isolated: $\Lambda_1$ is connected and contains three vertices, two of which are marked by $v_-$ (the top rightmost ``island'' in the state is the only loop which does not correspond to a vertex in $\Lambda_1$).
\end{example}

As this last comparison shows, there is still a lot of room for improvement in terms of detecting nontriviality of a state cycle via graph analysis.  Here are some more examples of this flavor, showing phenomena left to explain:

\begin{figure}[ht!]
\centering
\includegraphics{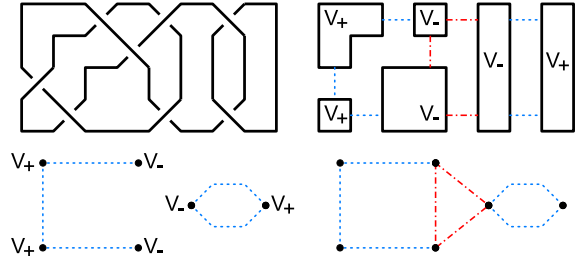}
\caption{Example of a 1-even state cycle which is not 1-isolated, and represents a nontrivial homology class.  On the top left is the diagram of the link; on the top right is the state cycle in question.  The bottom right shows the state graph for the state cycle, and the bottom left shows $\Lambda_1$, with the vertices marked as in the state cycle.  Clearly $\Lambda_1$ is 1-even, but since the lefthand component of $\Lambda_1$ has two loops marked by $v_-$, it is not 1-isolated.}
\label{figure-example_1even_nontrivial}
\end{figure}

\begin{example}
Often 1-even state cycles which are not 1-isolated still represent nontrivial homology classes.  An example is illustrated in Figure 
\ref{figure-example_1even_nontrivial}, which can be shown by direct calculation to represent a nontrivial homology class.  Presumably there is some condition on the 0-tracing loops that can extend this result to a larger class of 1-even state cycles.
\end{example}

\begin{figure}[ht!]
\centering
\includegraphics{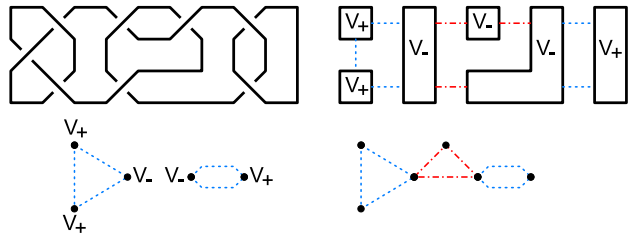}
\caption{Example of a 1-isolated state cycle which is not 1-even, and represents a nontrivial homology class.  On the top left is the diagram of the link; on the top right is the state cycle in question.  The bottom right shows the state graph for the state cycle, and the bottom left shows $\Lambda_1$, with the vertices marked as in the state cycle.  Clearly $\Lambda_1$ is 1-isolated, but since the lefthand component of $\Lambda_1$ has an odd length closed path, it is not 1-even.}
\label{figure-example_1isolated_nontrivial}
\end{figure}

\begin{example}
Similarly, it is possible for a 1-isolated state cycle which is not 1-even to represent a nontrivial homology class.  An example is illustrated in Figure 
\ref{figure-example_1isolated_nontrivial}, which can be shown by direct calculation to represent a nontrivial homology class.
\end{example}

\section{Constructing arbitrarily wide knots}

In this Section, we will describe how to use Theorem \ref{thm-1even_1isolated} to construct (hyperbolic) knots with nontrivial state cycles in arbitrarily many diagonals, and hence arbitrarily large homological width.  Since these knots will be prime, this phenomena is not simply a consequence of Khovanov's connect sum width formula; this technique gives a lower bound on width for some prime knots which quickly grow too large to calculate the full homology.

To do this, we will use knots with adequate, negative, nonalternating diagrams as building blocks.  We will ``connect'' them in some way to a + adequate base knot, in such a way that each copy of the block will give a nontrivial state cycle in a new diagonal, and the all-0 state remains adequate.  For example, 2 copies of the block, together with the base knot, will have 2 nontrivial state cycles from Theorem \ref{thm-1even_1isolated} together with the all-0 state cycle, for a minimum width of 3. 

\subsection{First family of examples}
\label{sub-first-family}

To start with, we will give a family of knots which may not all be hyperbolic, but have increasing width.  In practice, the first few members have turned out to be hyperbolic, but we have no technique to show this property holds for every member.  This problem will be remedied in the next subsection with a more complicated family of knots, based on this family.

For a + adequate base knot, we shall use the knot $8_{21}$, which is nonalternating and + adequate, but H-thin (width 2).  The + adequate diagram we will be using, as well as the associated all-0 state, is shown in Figure \ref{figure-8_21_base}.

\begin{figure}[ht!]
\centering
\includegraphics{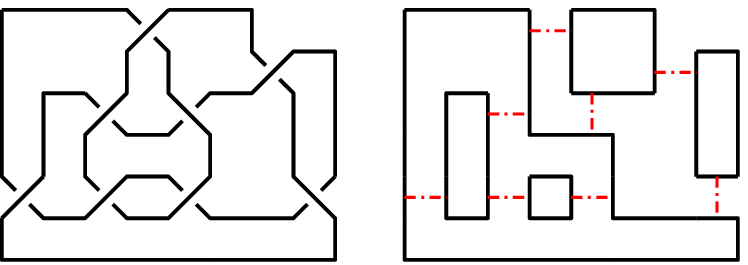}
\caption{A + adequate diagram for $8_{21}$ and its all-0 state.}
\label{figure-8_21_base}
\end{figure}

Our adequate, negative, nonalternating building block will be $10_{152}$.  The diagram we will be using, along with the respective all-0 and all-1 states, are shown in Figure \ref{figure-10_152_base}.  We will be combining copies of this diagram with the base knot in such a way that we can find nice 1-even, 1-isolated state cycles, which lie in different diagonals.

\begin{figure}[ht!]
\centering
\includegraphics{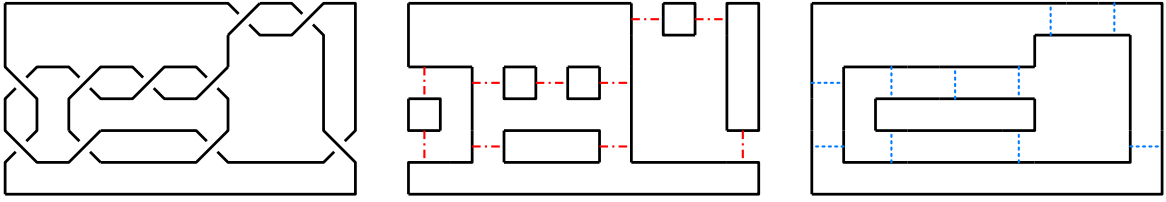}
\caption{An adequate, negative, nonalternating diagram for $10_{152}$, together with its all-0 and all-1 states.}
\label{figure-10_152_base}
\end{figure}

Note first that since this diagram of $10_{152}$ is negative, the all-1 state is its Seifert state, and hence is even.  Since the diagram is adequate, the all-0 and all-1 states are both adequate.  And, since the diagram is adequate and nonalternating, Khovanov (cf. Proposition 7 in ~\cite{khov-patterns}) has shown that the associated homology classes for the all-0 and all-1 states do not lie in adjacent diagonals.  This last property will be key in showing that the state cycles we construct will lie in distinct diagonals.

So, now we need to describe how to combine these pieces.  The idea is to entwine the block with the base knot so that the outer loop of the all-1 state is joined to a single loop of the base knot's all-0 state by positive crossings, in such a way that the diagram ``looks'' prime.  We indicate how we will do this entwining in Figure \ref{figure-entwining_first-family}, and denote such an entwining of $8_{21}$ with $n$ copies of $10_{152}$ by $K_n$.  Figure \ref{figure-entwining-8_21-2} illustrates $K_2$.

\begin{figure}[ht!]
\centering
\includegraphics{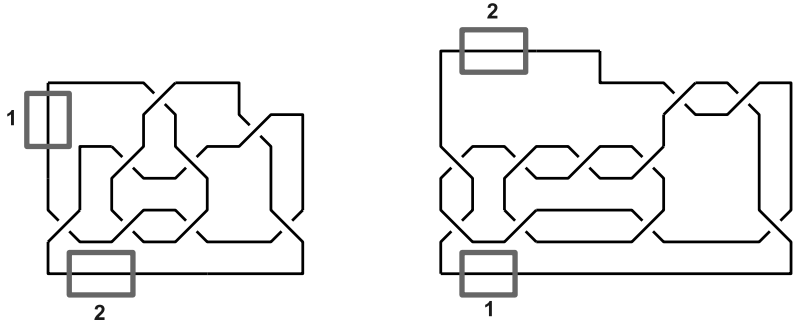}
\caption{Two regions are indicated on the base knot, and the building block knot.  The idea is to join region 1 by two positive half twists, and region 2 by three positive half twists.  See the next figure for what this looks like with two copies of the building block.}
\label{figure-entwining_first-family}
\end{figure}

\begin{figure}[ht!]
\centering
\includegraphics{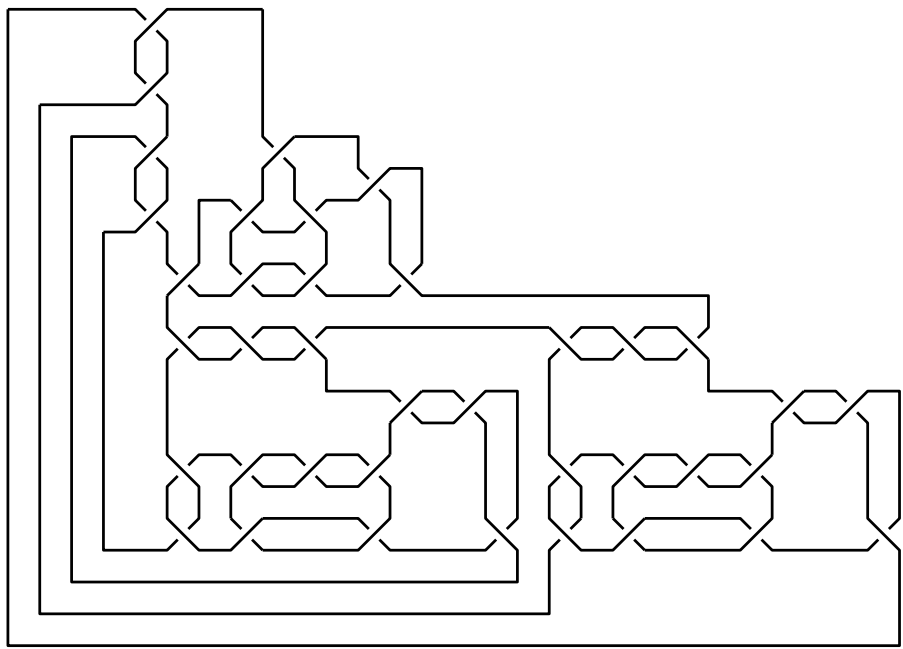}
\caption{The diagram for $K_2$, an entwining of the base knot $8_{21}$ with two copies of $10_{152}$.}
\label{figure-entwining-8_21-2}
\end{figure}

Now we need to find nontrivial state cycles in different diagonals.  Let $\alpha_0$ denote the state cycle where every crossing is 0-smoothed, and every loop is marked by $v_-$.  Pick some ordering of the building block parts in $K_n$, and let $\alpha_k$ be the state cycle where the first k building blocks are 1-smoothed, every other crossing is 0-smoothed, all 0-tracing loops are marked by $v_-$, and all other loops are marked by $v_+$.

\begin{prop}
Suppose $D_1$ is a + adequate diagram of a knot, and $D_2$ is an adequate, negative, nonalternating diagram of a prime knot.  Suppose that $K_n$ denotes one copy of $D_1$ joined to $n$ copies of $D_2$, so that each copy of $D_2$ has a single loop joined to $D_1$ by only positive crossings.  Then, with $\alpha_k$ as above, each $\alpha_k$ for $k=0, \ldots ,n$ is a nontrivial state cycle lying in a different diagonal.
\label{prop-entwining}
\end{prop} 

\begin{proof}
If $k=0$, then $\alpha_0$ is simply the all-0 state of a + adequate diagram, and hence represents a nontrivial homology class per Example \ref{figure-example_sigma0}.

So, suppose $k > 0$.  The connected components of $\Lambda_1$ for $\alpha_k$ consist of all the loops of each copy of $D_2$ which was 1-smoothed:  the only loops in the 1-block come from these copies, and each connected component of the 1-block is connected to only one other loop by 1-traces: the loop from $D_2$ which is joined to $D_1$ by positive crossings (0-smoothed in $\alpha_k$).  Each such component is even because the all-1 state of a negative diagram is its Seifert state.  Therefore, $\alpha_k$ is 1-even.

The single loop of each 1-smoothed copy  of $D_2$ which is joined to $D_1$ by positive crossings is now 0-tracing in $\alpha_k$, and so each connected component of $\Lambda_1$ has only a single loop marked by $v_-$.  So, $\alpha_k$ is also 1-isolated, and hence represents a nontrivial homology class by Theorem \ref{thm-1even_1isolated}.

The reason each $\alpha_k$ lies in a different diagonal is roughly the same as Khovanov's proof that knots with adequate, nonalternating diagrams are H-thick (Proposition 7 of ~\cite{khov-patterns}).  To prove this explicitly, we will show that $\alpha_k$ lies in a diagonal below $\alpha_{k-1}$.

The only difference between $\alpha_k$ and $\alpha_{k-1}$ is that the $k$-th copy of $D_2$ is 1-smoothed in $\alpha_k$, and 0-smoothed in $\alpha_{k-1}$.  So, it suffices to check how the relative diagonal gradings of these two parts compare.  Furthermore, the writhe parts of the homological and quantum gradings for these two parts will remain the same, so we can restrict our attention to the relative homological and quantum gradings, to see how their relative diagonal gradings compare.

To fix notation, let $\sigma_0$ be the state cycle for the all-0 state of $D_2$, with all loops marked by $v_-$.  Let $\sigma_1$ be the state cycle for the all-1 state of $D_2$, where the loop joined to $D_1$ in $K_n$ is marked by $v_-$, and all other loops are marked by $v_+$.  Let $s_0, s_1$ denote the number of loops in $\sigma_0, \sigma_1$, and $m$ be the number of crossings in $D_2$.

Since we are ignoring the $n_+, n_-$ parts of the gradings, the relative homological grading will simply be the number of 1-smoothings, so that $t(\sigma_0) = 0, t(\sigma_1) = m$.  The relative quantum gradings will be given by $q(\sigma_0) = -s_0 + 0 = -s_0$, and $q(\sigma_1) = (s_1 - 1) - 1 + m = s_1 + m - 2$.  The (relative diagonal) $\delta$ gradings are $2t - q$, so $\delta(\sigma_0) = 2* 0 - (-s_0) = s_0$, and $\delta(\sigma_1) = 2m - (s_1 + m -2) = m - s_1  + 2$.  So, for these two state cycles to lie on the same diagonal, we would have to have $s_0 = m - s_1 + 2$, i.e. $s_0 + s_1 = m+2$.

However, because $D_2$ is an adequate, nonalternating diagram of a prime knot, a lemma from Chapter 5 of ~\cite{lickorish-knottheory} states that $s_0 + s_1 < m+2$.  So, $\sigma_0$ and $\sigma_1$ have different relative $\delta$ gradings, and must lie on different diagonals:  the inequality guarantees that $\sigma_0$ lies in a diagonal lower than $\sigma_1$.  This same difference of $\delta$ gradings holds for $\alpha_k$ and $\alpha_{k-1}$, so these also lie on different diagonals, with $\alpha_{k-1}$ lying below $\alpha_k$.  It follows that each $\alpha_k$ lies on its own diagonal, as claimed.  
\end{proof}

\begin{cor}
For such a family of knots $K_n$, the homological width of $K_n$ is at least $n+1$.  In particular, this holds when $D_1 = 8_{21}, D_2 = 10_{152}$.
\end{cor}

\begin{figure}[ht!]
\centering
\includegraphics{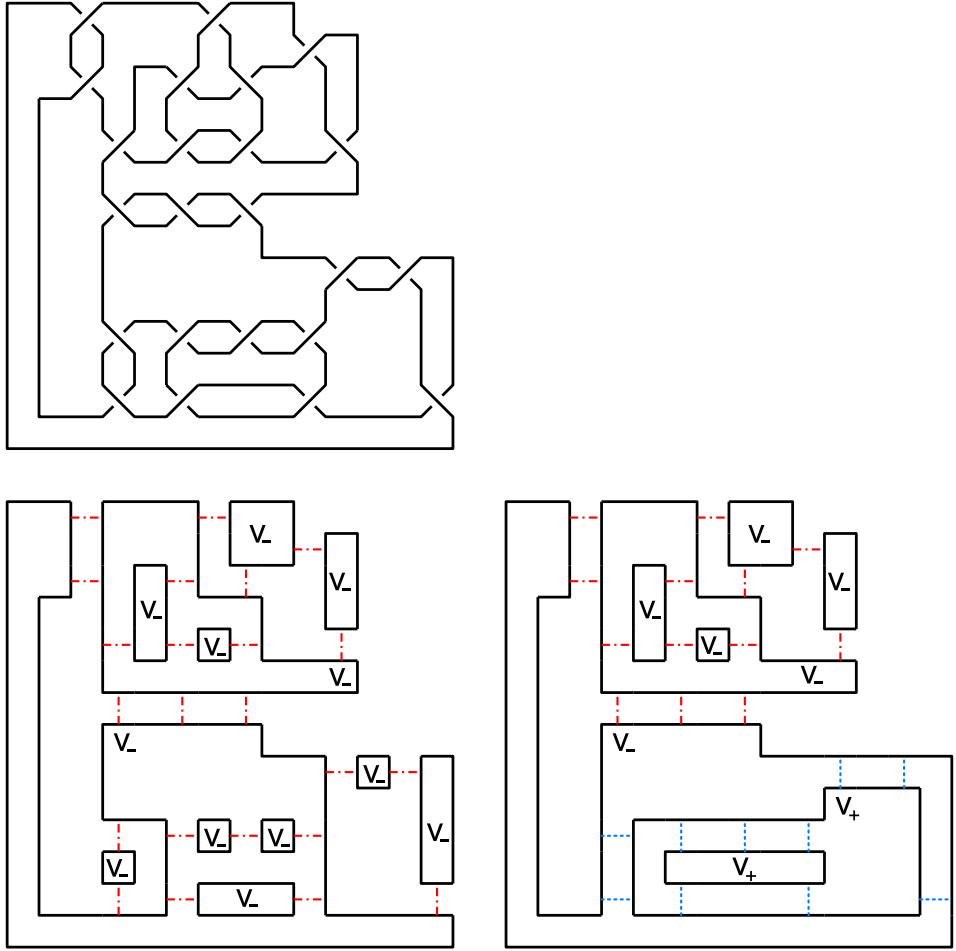}
\caption{On top is the diagram for $K_1$, an entwining of the base knot $8_{21}$ with one copy of $10_{152}$.  On the bottom, from left to right we have the state cycles for $\alpha_0, \alpha_1$.}
\label{figure-entwining-8_21-1}
\end{figure}

Unfortunately, this lower bound is not always reflect the full width of a knot; we can see this in the case of $K_1$, whose width turns out to be 4, instead of the 2 this bound gives.  Figure \ref{figure-entwining-8_21-1} shows the diagram for this knot, as well as $\alpha_0$ and $\alpha_1$.  The diagram has 7 positive crossings and 16 negative crossings, and one can check that the bigradings of $\alpha_0, \alpha_1$ are $(-16,-37)$ and $(-6,-19)$, and the $\delta$ gradings are 5 and 7.  Using Morrison's JavaKh-v2 ~\cite{fgmw-spc4}, we calculated the Khovanov homology, and organized the ranks into Table \ref{table-K_1-homology}.  The table makes it clear that the homological width of $K_1$ is 4; the top two diagonals are those from $\alpha_0, \alpha_1$.  We attempted to check this for $K_2$, but the calculation ran out of memory on the 4 GB desktops we were working with.  We expect the lower bound again to be lower than the actual width for this larger case.

\begin{table}[htbp]
\tiny
\begin{center}
\setlength\arrayrulewidth{1pt}
\setlength\doublerulesep{1pt}
\begin{tabular}{|q||q|q|q|q|q|q|q|q|q|q|q|q|q|q|q|q|q|q|q|q|q|q|} \hline
 & -16 & -15 & -14 & -13 & -12 & -11 & -10 & -9 & -8 &  -7 & -6 & -5 & -4 & -3 & -2 & -1 & \phantom{-}0 & \phantom{-}1 & 2 & 3 & 4 & 5 \\ \hline\hline
\phantom{-0}3 	&  &   	 &   	 &  	 &  	 &  	 &  &	&	&	&	&	&	&	&	&	&	&	&	&{\cy}	&{\cb}	&  {\cy}1 \\ \hline
\phantom{-0}1 	&  &   	 &   	 &  	 &  	 &  	 &  &	&	&	&	&	&	&	&	&	&	&	&{\cy}	& {\cb}1	&	{\cy}& {\cb}  \\ \hline
\phantom{0}-1 	&  &   	 &   	 &  	 &  	 &  	 &  &	&	&	&	&	&	&	&	&	&	&{\cy}	& {\cb}5	& {\cy}1	& {\cb}1	&   \\ \hline
\phantom{0}-3 	&  &   	 &   	 &  	 &  	 &  	 &  &	&	&	&	&	&	&	&	&	&	{\cy}& {\cb}20	& {\cy}2	& {\cb}	& 	&   \\ \hline
\phantom{0}-5 	&  &   	 &   	 &  	 &  	 &  	 &  &	&	&	&	&	&	&	&	&{\cy}	& {\cb}47	& {\cy}5	& {\cb}1	& 	& 	&   \\ \hline
\phantom{0}-7 	&  &   	 &   	 &  	 &  	 &  	 &  &	&	&	&	&	&	&	&{\cy}	& {\cb}86	& {\cy}22	& {\cb}1	& 	& 	& 	&   \\ \hline
\phantom{0}-9 	&  &   	 &   	 &  	 &  	 &  	 &  &	&	&	&	&	&	& {\cy}	& {\cb}136	& {\cy}46	& {\cb}	& 	& 	& 	& 	&   \\ \hline
-11 	&  &   	 &   	 &  	 &  	 &  	 &  &	&	&	&	&	&	{\cy} & {\cb}174	& {\cy}86	& {\cb}1	& 	& 	& 	& 	& 	&   \\ \hline
-13 	&  &   	 &   	 &  	 &  	 &  	 &  &	&	&	&	&{\cy}	& {\cb}204	& {\cy}136	& {\cb}	& 	& 	& 	& 	& 	& 	&   \\ \hline
-15 	&  &   	 &   	 &  	 &  	 &  	 &  &	&	&	&{\cy}	& {\cb}204	& {\cy}174	& {\cb}	& 	& 	& 	& 	& 	& 	& 	&   \\ \hline
-17 	&  &   	 &   	 &  	 &  	 &  	 &  &	&  & {\cy}1	& {\cb}181	& {\cy}204	& {\cb}	& 	& 	& 	& 	& 	& 	& 	& 	&   \\ \hline
-19 	&  &   	 &   	 &  	 &  	 &  	 &  &	& {\cy}3 & {\cb}143	& {\cy}204	& {\cb}	& 	& 	& 	& 	& 	& 	& 	& 	& 	&   \\ \hline
-21 	&  &   	 &   	 &  	 &  	 &  	 &  & {\cy}4	& {\cb}94 & {\cy}181	& {\cb}	& 	& 	& 	& 	& 	& 	& 	& 	& 	& 	&   \\ \hline
-23 	&  &   	 &   	 &  	 &  	 &  	 &  {\cy}5 & {\cb}54	& {\cy}143 & {\cb}	& 	& 	& 	& 	& 	& 	& 	& 	& 	& 	& 	&   \\ \hline
-25 	&  &   	 &   	 &  	 &  	 &  	{\cy}6 &  {\cb}25 & {\cy}93	& {\cb} & 	& 	& 	& 	& 	& 	& 	& 	& 	& 	& 	& 	&   \\ \hline
-27 	&  &   	 &   	 &  	 &  {\cy}5	 &  	{\cb}9 &  {\cy}51 & {\cb}	&  & 	& 	& 	& 	& 	& 	& 	& 	& 	& 	& 	& 	&   \\ \hline
-29 	&  &   	 &   	 &  {\cy}4	 &  {\cb}6	 &  	{\cy}21 & {\cb}  & 	&  & 	& 	& 	& 	& 	& 	& 	& 	& 	& 	& 	& 	&   \\ \hline
-31 	&  &   	 &   {\cy}3	 &  {\cb}5	 &  {\cy}4	 &  {\cb}	 &   & 	&  & 	& 	& 	& 	& 	& 	& 	& 	& 	& 	& 	& 	&   \\ \hline
-33 	&  & {\cy}  1	 &   {\cb}4	 &  {\cy}	 &  {\cb}	 &  	 &   & 	&  & 	& 	& 	& 	& 	& 	& 	& 	& 	& 	& 	& 	&   \\ \hline
-35 	& {\cy} &   {\cb}3	 & {\cy}  	 & {\cb} 	 &  	 &  	 &   & 	&  & 	& 	& 	& 	& 	& 	& 	& 	& 	& 	& 	& 	&   \\ \hline
-37 	&  {\cb}1 &  {\cy} 	 &   {\cb}	 &  	 &  	 &  	 &   & 	&  & 	& 	& 	& 	& 	& 	& 	& 	& 	& 	& 	& 	&   \\ \hline
\end{tabular}
\end{center}
\caption{Rational Khovanov homology of $K_1$, with the four diagonals shaded.}
\label{table-K_1-homology}
\end{table}

We suspect that the $K_n$ in this section are prime, and probably hyperbolic.  But, we did not find a recursive prime tangle decomposition for the knots, and while we were able to check that $K_0, K_1, K_2$ are hyperbolic using SnapPea ~\cite{snappea}, we do not know any nice theorem to show this holds in general.

\subsection{A hyperbolic family}
\label{sub-hyperbolic-family}

In order to guarantee that we have a family of knots with every member hyperbolic, we will modify the construction of \ref{sub-first-family} so that we may apply a theorem of Futer and Purcell.  To state their theorem, we will first need to review some terminology.  For more details, we refer the reader to ~\cite{lackenby-halc, futer-purcell-2007}.

One can view a diagram $D$ as a 4-valent planar graph $\Gamma_4(D)$, where the vertices are colored by over and under crossing information.  A \emph{twist region} of $\Gamma_4(D)$ is a maximum string of bigons strung together, vertex to vertex; one also allows an empty collection of bigons for a twist region, if a crossing is adjacent to no bigons.  Intuitively, a twist region is just a section of the diagram where two strands are tied into a series of half-twists.

A diagram $D$ of a link is called \emph{prime} if, for any simple curve $\gamma$ that intersects $\Gamma_4(D)$ transversely in two points, $\gamma$ bounds a subdiagram containing no crossings.  Roughly, this means the diagram is prime if it is not visibly a connect sum.

A diagram $D$ is \emph{twist-reduced} if, whenever a
simple closed curve $\gamma$ intersects the graph $\Gamma_4(D)$ transversely
in four points, with two points adjacent to one crossing and
the other two points adjacent to another crossing, then $\gamma$ bounds a twist region.  See Figure \ref{figure-example-twist_reduced} for examples and nonexamples of such curves $\gamma$.  Given a $\gamma$ from the definition, this rules out the case that one strand might be knotted, or knotted and linked with only the other strand contained in the curve.

\begin{figure}[ht!]
\centering
\includegraphics{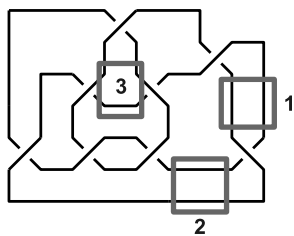}
\caption{Examples and nonexamples of valid curves $\gamma$ for the twist-reduced definition.  Curve 1 is valid, with the top two intersection points near one crossing, and the bottom two near a second.  Curve 2 is not valid; the two intersection points on the right are adjacent to a crossing, but the pair on the left is not adjacent to a single crossing.  Curve 3 is also not valid; the pair of points on the right and left are each adjacent to different crossings, but the curve intersects the diagram transversely in six points, not four.  One can verify that this diagram is in fact twist-reduced.}
\label{figure-example-twist_reduced}
\end{figure}

With these definitions in mind, the theorem we wish to invoke is the following:

\begin{thm}[~\cite{futer-purcell-2007}]
Let $L \in S^3$ be a link with a prime, twist-reduced diagram $D$.
Assume that $D$ has at least two twist regions.  If every twist region of $D$ contains at least 6 crossings, then $L$ is hyperbolic.
\label{thm-fp}
\end{thm}

In other words, we are going to add a bunch of extra twists to each twist region of the base knot $D_1$ and the building block $D_2$, and combine them into diagrams which are twist-reduced and prime, to conclude that the knots in question are hyperbolic (and hence prime).  So, let's do this for the case when $D_1 = 8_{21}$ and $D_2 = 10_{152}$.  Figure \ref{figure-twist_regions} has each twist region of these diagrams circled; we will then make new diagrams $D_1', D_2'$ which have at least six crossings in each twist region, as indicated in Figure \ref{figure-entwining_family-two}.

\begin{figure}[ht!]
\centering
\includegraphics{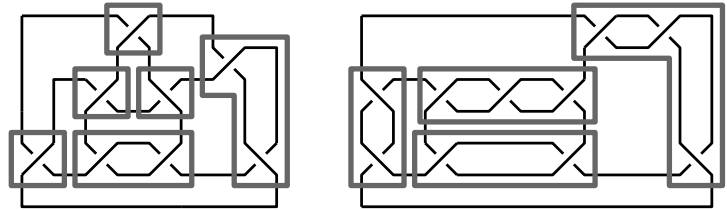}
\caption{Diagrams for $8_{21}, 10_{152}$ with twist regions indicated.  One can check that both diagrams are twist-reduced.}
\label{figure-twist_regions}
\end{figure}

One can verify that the original diagrams for $D_1$ and $D_2$ as given in Figure \ref{figure-twist_regions} are twist-reduced, as the only valid $\gamma$ from the twist-reduced definition occur in these twist regions; adding extra twists within
each twist region does not change this property, so $D_1, D_2$ are also twist-reduced.

\begin{figure}[ht!]
\centering
\includegraphics{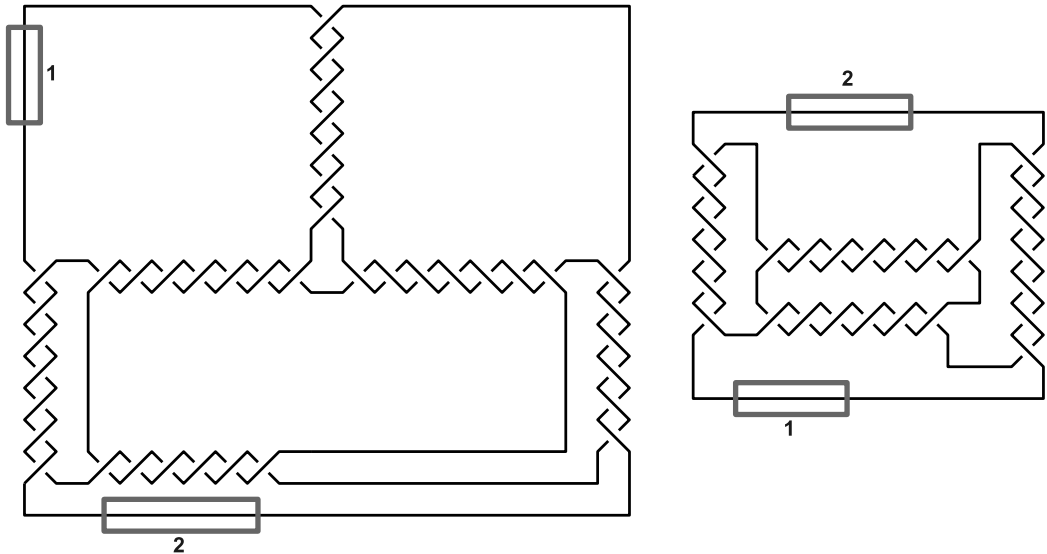}
\caption{Diagrams for $D_1', D_2'$, where each twist region from Figure \ref{figure-twist_regions} has been replaced by six or seven crossings, depending on the parity crossings in the original twist regions.  Two regions have been marked on each diagram to indicate how to combine the two into $K_n'$:  the pair of strands marked for region 1 will be joined by six positive crossings, while the pair of strands marked for region 2 will be joined by seven positive crossings.}
\label{figure-entwining_family-two}
\end{figure}

Going from the twist-reduced definition, the entwining pattern as shown in Figure \ref{figure-entwining_family-two} will introduce two new twist regions between each copy of $D_2$ in $K_n'$; one can check that this entwining pattern will always result in twist-reduced and prime diagrams.  By construction these diagrams have more than two twist regions, and each twist region has at least 6 crossings in it, so Theorem \ref{thm-fp} applies, and the resulting $K_n'$ are hyperbolic, as claimed.

If one is worried about checking all the valid curves to show these knots are twist-reduced, there is an alternate approach which works directly from the augmented links.  The idea is to come up with a triangulation of $S^2$ based on the augmented link which satisfies Andreev's theorem (Theorem 6.2 of ~\cite{purcell-cusps}).  With such a triangulation, Purcell shows the augmented link is hyperbolic, at which point one can follow the rest of the proof that the original knot is hyperbolic from ~\cite{futer-purcell-2007}.

The \emph{augmented link} of a diagram comes by replacing each twist region of the diagram by a surgery circle around the two strands being twisted in the twist region. The triangulation in question is formed by adding a vertex for each region in the diagram of the augmented link, where such a region is marked off by arcs of the diagram or a surgery circle in the augmented link.  We then add an edge between two vertices if they share a common arc in the diagram, or along a surgery circle if one spans the two regions.  The triangulations for the augmented links for $D_1, D_2$ are shown in Figure \ref{figure-8_21_10_152-twist-reduced}.

\begin{figure}[ht!]
\centering
\includegraphics{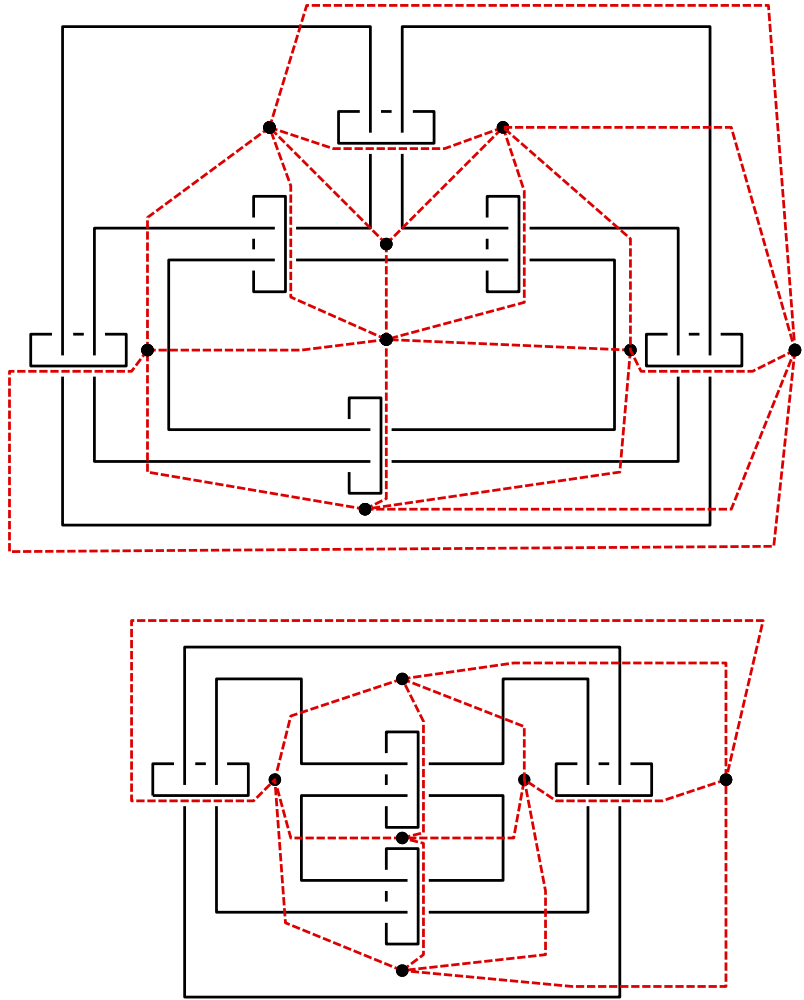}
\caption{Augmented link diagrams for $8_{21}, 10_{152}$, with the associated triangulation marked in red dashed lines.  Because each edge has exactly two vertices, and there are no two vertices joined by more than one edge, these two triangulations satisfy Andreev's theorem.  So, by Purcell's argument, the augmented links are hyperbolic.}
\label{figure-8_21_10_152-twist-reduced}
\end{figure}

Given such a triangulation, the two conditions we need to apply Andreev's theorem are:
\begin{itemize}
\item Each edge has distinct ends.
\item No two vertices are joined by more than one edge.
\end{itemize}
In other words, the constructed graph provides an honest triangulation of $S^2$.  Glancing back at Figure \ref{figure-8_21_10_152-twist-reduced}, one can easily check that these two conditions hold for the triangulations of the augmented links for $D_1', D_2'$.

One then must check that such a triangulation continues to exist when we start piecing together the knots of our family.  A schematic for the augmented link diagram for $K_n'$ is given in Figure \ref{figure-family_two-twist-reduced}.  Reading the schematic, we have one copy of $D_1'$ and $n$ copies of $D_2'$ being glued together in a proscribed fashion.  The ending vertices of edges which go between the copies of $D_1'$ and $D_2'$ are the $A(i), B(i), C(i), D(i), E(i), F, G$.  One can verify that the new edges being added in this gluing still yield triangulations which satisfy Andreev's theorem, so that the augmented link for $K_n'$ is hyperbolic, as desired.


\begin{figure}[ht!]
\centering
\includegraphics{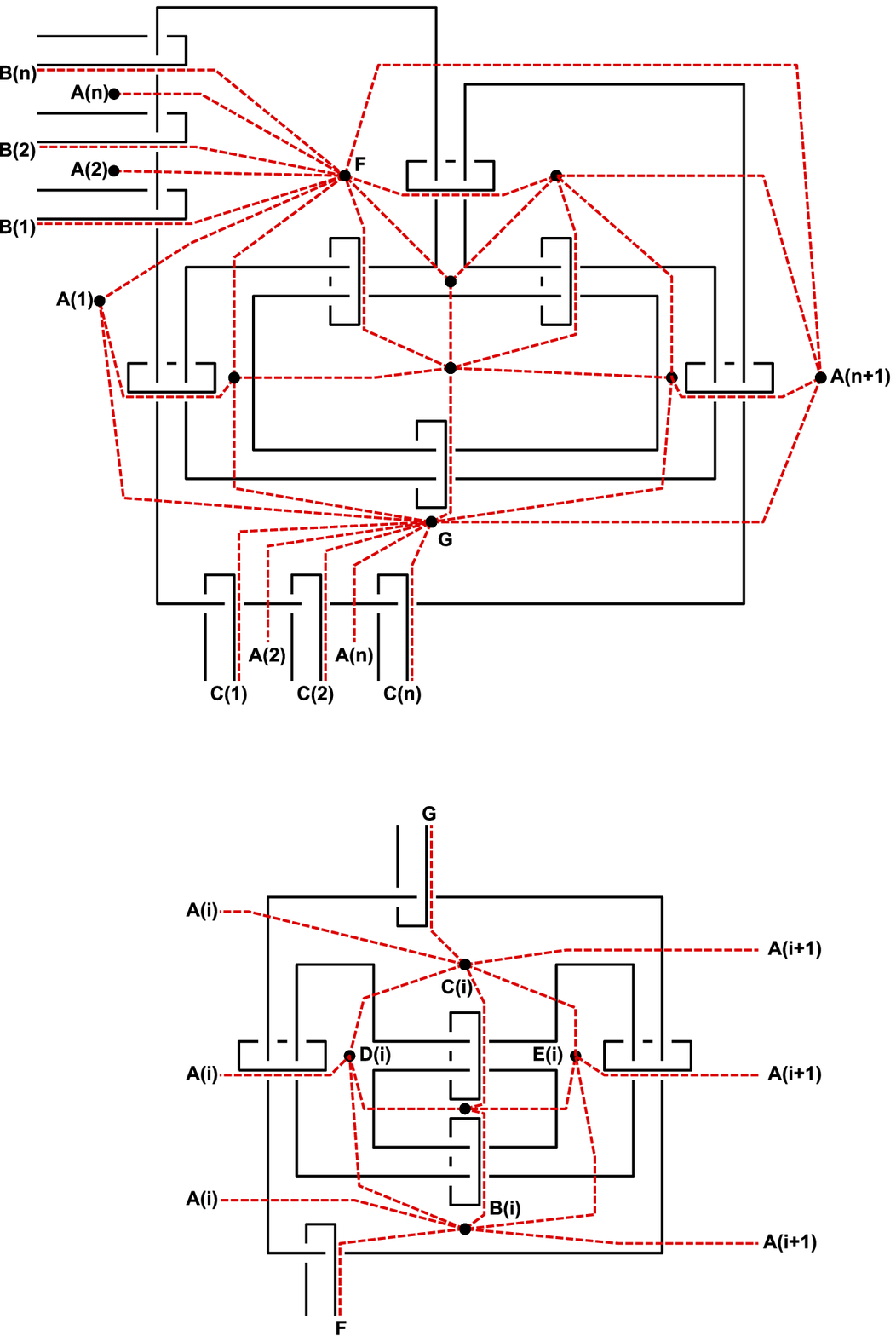}
\caption{Schematics for the augmented link diagram of $K_n'$.  On the top is the augmentation for $D_1'$, while on the bottom is an augmentation schematic for the $i$-th copy of $D_2'$.  Dotted edges which go between the two diagrams are labelled as to their ending vertices.}
\label{figure-family_two-twist-reduced}
\end{figure}

Because we simply duplicated the twist type for each twist region, and maintained the parity of the number of half twists, $D_1'$ remains + adequate, and $D_2'$ remains nonalternating, adequate, and negative.  So, the entwining pattern satisfies Proposition \ref{prop-entwining}, and each $K_n'$ contains at least $n+1$ nontrivial state cycles, each in its own diagonal.  Thus, this family of examples demonstrates how to produce hyperbolic knots of arbitrarily large homological width.
\clearpage

\bibliographystyle{hep}
\bibliography{knotbib}
\end{document}